\documentclass[3p]{article}
\usepackage{amsfonts,amsmath,amssymb,amsthm,stmaryrd,hyperref,bbm}
\usepackage[francais,english]{babel}
\usepackage[T1]{fontenc}
\usepackage[ansinew]{inputenc}  
\usepackage{graphics}
\usepackage{graphicx}
\usepackage{tikz}
\usepackage{multicol}
\usetikzlibrary{patterns}
\usepackage{amssymb,subfigure,lineno,url}
\usepackage{float}
\usepackage[draft]{showlabels}
\newcommand{\vertiii}[1]{{\left\vert\kern-0.25ex\left\vert\kern-0.25ex\left\vert #1 
    \right\vert\kern-0.25ex\right\vert\kern-0.25ex\right\vert}}
    
\newtheorem{theorem}{Theorem}
\newtheorem{definition}{Definition}
\newtheorem{lemma}{Lemma}
\newtheorem{proposition}{Proposition}
\newtheorem{remark}{ Remark}

\begin{document}
\title{Cazenave-Dickstein-Weissler-type extension of Fujita's problem on Heisenberg groups}

\author{Mokhtar Kirane \footnote{\noindent Department of Mathematics, Faculty of Arts and Science, Khalifa University, P.O. Box: 127788,  Abu Dhabi, UAE; Email: mokhtar.kirane@ku.ac.ae
 \newline \indent $\,\,{}^{a}$Ahmad Z. Fino, College of Engineering and Technology, American University of the Middle East, Kuwait; Email: ahmad.fino@aum.edu.kw
   \newline \indent $\,\,{}^{b}$ Berikbol T. Torebek, Department of Mathematics: Analysis, Logic and Discrete Mathematics, Ghent University, Belgium. 
   \newline \indent $\,\,\,$ Institute of Mathematics and Mathematical Modeling, 125 Pushkin str., Almaty 050010, Kazakhstan; Email: berikbol.torebek@ugent.be
     \newline \indent $\,\,{}^{c}$ Zineb Sabbagh, Department of Mathematics, 
University Saad Dahleb, Blida, Algeria; Email: szinebedp@gmail.com
 }, Ahmad Z. Fino$^{\,a}$, Berikbol T. Torebek$^{\,b}$, Zineb Sabbagh$^{\,c}$}

\date{}
\maketitle

\begin{abstract} 
This paper examines the critical exponents for the existence of global solutions to the equation
\begin{equation*}
\begin{array}{ll}
\displaystyle u_t-\Delta_{\mathbb{H}}u=\int_0^t(t-s)^{-\gamma}|u(s)|^{p-1}u(s)\,ds,&\qquad 0\leq\gamma<1,\,\,\, {\eta\in \mathbb{H}^n,\,\,\,t>0,}
 \end{array}
\end{equation*} on the Heisenberg groups $\mathbb{H}^n.$ There exists a critical exponent $$p_c= \max\Big\{\frac{1}{\gamma},\frac{n+2}{n+\gamma}\Big\}\in(0,+\infty],$$ such that for all $1<p\leq p_c,$ no global solution exists regardless of the non-negative initial data, while for $p>p_c$, a global positive solution exists if the initial data is sufficiently small. The results obtained are a natural extension of the results of Cazenave et al. [Nonlinear Analysis 68 (2008), 862-874], where similar studies were carried out in $\mathbb{R}^n$. Furthermore, several theorems are presented that provide lifespan estimates for local solutions under various initial data conditions. The proofs of the main results are based on test function methods and Banach fixed point principle.
\end{abstract}

\noindent {\small {\bf MSC[2020]:} 35A01; 35R03; 35B53; 35K65; 35B33} 

\noindent {\small {\bf Keywords:} Critical exponents, semilinear parabolic equations, Heisenberg group.}


\section{Introduction}

In this paper, we examines the global existence of solutions of the following problem
\begin{equation}\label{1}
\begin{array}{ll}
\displaystyle u_t-\Delta_{\mathbb{H}}u=\int_0^t(t-s)^{-\gamma}|u(s)|^{p-1}u(s)\,ds,&\qquad {\eta\in \mathbb{H}^n,\,\,\,t>0,}
 \\{}\\ u(0,\eta)=u_{0}(\eta), &\qquad \eta\in \mathbb{H}^n,
 \end{array}
\end{equation}
where $n\geq1$, $p>1$, $0\leq \gamma<1$, $u_0\in C_0(\mathbb{H}^n)$, $\Delta_{\mathbb{H}}$ is the Heisenberg Laplacian, and  $C_0(\mathbb{H}^n)$ denotes the space of all continuous functions tending to zero at infinity. 

The problem \eqref{1} is the setting of the semilinear heat equation with nonlocal nonlinearity
\begin{equation}\label{2}
\begin{array}{ll}
\displaystyle u_t-\Delta u=\int_0^t(t-s)^{-\gamma}|u(s)|^{p-1}u(s)\,ds,&\qquad {x\in \mathbb{R}^n,\,\,\,t>0,}\end{array}
\end{equation} which Cazenave et al. \cite{CDW} studied in the Euclidean space $\mathbb{R}^n$, on the Heisenberg group. Here $\Delta$ is a Laplacian in $\mathbb{R}^n$.
In particular, in \cite{CDW} it is proved that the critical exponent ensuring global existence of solutions to equation \eqref{2} is given by 
$$p_{\mathrm{crit}}=\max\Big\{\frac{1}{\gamma},1+\frac{2(2-\gamma)}{(n-2+2\gamma)_+}\Big\}\in(0,+\infty],$$
 i.e., it has been proven that: if $p \leq p_{\mathrm{crit}}$, and $u_0$ is nonnegative, then $u$ blows up in finite time, while there exist a global positive solution for sufficiently small $u_0\geq 0$, if $p>p_{\mathrm{crit}}.$

This in turn coincides with Fujita's pioneering results \cite{Fujita}, where a critical exponent of the form $$p_F=1+\frac{2}{n}$$ 
was established for the classical semilinear heat equation
 \begin{equation}\label{3}
 u_t=\Delta u +u^p,\qquad \text{in}\,\,\mathbb{R}^n.
 \end{equation}
Notice that when $\gamma\to1$, using the relation
$$\lim_{\gamma\to1}c_\gamma\,s_+^{-\gamma}=\delta_0(s)\quad\text{in distributional sense with}\,\,s_+^{-\gamma}:=\left\{\begin{array}{ll}
s^{-\gamma}&\,\,\text{if}\,\,s>0,\\
0&\,\,\text{if}\,\,s<0,
\end{array}
\right.
$$
where $c_\gamma=1/\Gamma(1-\gamma)$, along with a suitable change of variables, we see that problem \eqref{2} reduces to the semilinear heat equation \eqref{3}. Thus, the critical exponent $\lim\limits_{\gamma\rightarrow1}p_{\mathrm{crit}}=1+\frac{2}{n}=p_F$ for problem \eqref{2} matches the Fujita critical exponent for problem \eqref{3}.

To our knowledge, Fujita-type results on Heisenberg groups $\mathbb{H}^n$ were first obtained by Zhang \cite{Zhang} and Pascucci \cite{Pascucci}. In the aforementioned studies, the critical Fujita exponent $p_c=1+\frac{2}{Q}$ was derived, where $Q=2n+2$ is the homogeneous dimension of the Heisenberg groups $\mathbb{H}^n$.

Further extensions of Fujita-type results on the Heisenberg group have been obtained by many authors (see for example \cite{Kirane2, Torebek, D'Ambrosio, FRT, Georgiev, Kirane3, Kirane4, Kirane5, Pohozaev} and the references therein). More specifically, the first two authors of this paper, together with their collaborators \cite{FKBK}, have recently examined a problem even closer to \eqref{1}: They analyzed Fujita-type results for semilinear heat inequality with convolution-type nonlinearity:
\begin{align*}
u_t-\Delta_{\mathbb{H}}u\geq \left(\mathcal{K}*_{\mathbb{H}}|u|^{p}\right)|u|^q,&\qquad {\eta\in \mathbb{H}^n,\,\,\,t>0,}
\end{align*} where $\mathcal{K}:(0,\infty)\rightarrow(0,\infty)$ is a continuous function satisfying $\mathcal{K}\left(|\cdot|_{\mathbb{H}}\right)\in L^1_{loc}(\mathbb{H}^n)$ which decreases in a vicinity of infinity, and $*_{\mathbb{H}}$ denotes the convolution operation in $\mathbb{H}^n.$
In particular, for $\mathcal{K}\left(r\right) = r^{-\alpha},\,\alpha\in(0,Q)$, it has been shown that no global solutions exist whenever $$\int_{\mathbb{H}^n}u_0(\eta)d\eta>0\,\,\,\,\text{and}\,\,\,\,2<p + q <\frac{2(Q+1)}{Q+\alpha}.$$
To the best of our knowledge, no Fujita-type analyses have yet been carried out for semilinear heat equations with time-nonlocal nonlinearities; this gap in the literature prompts us to investigate the associated Fujita critical exponents for problem \eqref{1}.

For the convenience of the readers, we present the main results of the paper in the subsection below, and then in the following sections we present the proofs of the main results in turn.

\subsection{Main results}
\begin{definition}[Mild solution]${}$\\
Let $u_0\in C_0(\mathbb{H}^n)$, $n\geq1$, $p>1$, and $T>0$. We say that $u\in C([0,T],C_0(\mathbb{H}^n))$
is a mild solution of problem \eqref{1} if $u$ satisfies the following integral equation
\begin{equation}\label{IE}
    u(t,\eta)=S_{\mathbb{H}}(t) u_0(\eta)+\Gamma(\alpha)\int_{0}^tS_{\mathbb{H}}(t-s)I_{0|s}^\alpha(|u|^{p-1}u)(s,\eta)\,ds,\quad \hbox{for all}\,\,\,\eta\in \mathbb{H}^n,\,t\in[0,T],
\end{equation}
where $\alpha=1-\gamma$, $\Gamma$ is the Euler gamma function and $I_{0|s}^\alpha$ is the Riemann-Liouville left-sided fractional integral defined in \eqref{I1} below. More general, for all $t_0\geq 0$, we have
\begin{equation}\label{IEG}
    u(t,\eta)=S_{\mathbb{H}}(t-t_0) u(t_0,\eta)+\Gamma(\alpha)\int_{t_0}^tS_{\mathbb{H}}(t-s)I_{0|s}^\alpha(|u|^{p-1}u)(s,\eta)\,ds,\quad \hbox{for all}\,\,\,\eta\in \mathbb{H}^n,\,t\in[t_0,T].
\end{equation}
If $u$ is a mild solution of \eqref{1} in $[0,T]$ for all $T>0$, then $u$ is called global-in-time mild solution of \eqref{1}.
\end{definition}
First, we establish the existence and uniqueness of a local mild solution.
\begin{theorem}[Local existence]\label{Local}${}$\\
Let $u_0\in C_0(\mathbb{H}^n)$, with $n\geq1$ and $p>1$. Then there exists a maximal
time $T_{\max}>0$ and a unique mild solution $u\in
C([0,T_{\max}),C_0(\mathbb{H}^n))$ to problem \eqref{1}. Moreover, we have the alternative:\\
\begin{itemize} 
\item[$\bullet$] Either $T_{\max}=+\infty$, in which case the solution is global, or
\item[$\bullet$]  $T_{\max}<+\infty$ and 
$$\liminf_{t\rightarrow T_{\max}}\|u\|_{L^\infty((0,t),L^\infty(\mathbb{H}^{n}))}=+\infty,$$
i.e., the solution blows up in finite time.
\end{itemize} 
 In addition, if $u_0\geq0,$ $u_0\not\equiv0,$ then the solution remains strictly positive for all $0<t<T_{\max}.$
 Furthermore, if $u_0\in L^r(\mathbb{H}^n),$ for some $1\leq r<\infty,$ then $u\in C([0,T_{\max}),L^r(\mathbb{H}^n))$.
\end{theorem}
\begin{remark}
Theorem \ref{Local} remains valid for a more general source term, provided that $|u|^{p-1}u$ is replaced by an arbitrary locally Lipschitz function $f(u)$.
\end{remark}
Next, we establish conditions that guarantee the global existence of mild solutions and, conversely, delineate the parameter regimes in which these solutions undergo finite-time blow-up. For the formulation of the upcoming theorems, recall that $Q=2n+2$ denotes the homogeneous dimension of the Heisenberg group.
\begin{theorem}\label{Blowup-global} Let $0\leq\gamma<1$, $p>1$, and set
\[p_c= \max\Big\{\frac{1}{\gamma},p_\gamma\Big\}\in(0,+\infty],\qquad\hbox{with}\quad p_\gamma=1+\frac{2(2-\gamma)}{Q-2+2\gamma}.\]
Let $u_0\in C_0(\mathbb{H}^n)$, and let $u\in
C([0,T_{\max}),C_0(\mathbb{H}^n))$ be the corresponding solution of \eqref{1}.  If $\gamma=0$, then all nontrivial positive solutions blow up. If $0<\gamma<1$, we have the following:
\begin{enumerate}
\item[$(\mathrm{i})$] If
 \begin{equation}\label{esti1}
    p\leq p_c,
\end{equation}
and $u_0\geq0$, $u_0\not\equiv 0$, then the mild solution $u$ of \eqref{1} blows up in finite time.\\
\item[$(\mathrm{ii})$] If
 \begin{equation}\label{esti2}
    p> p_c,
\end{equation}
and $u_0\in L^{q_{\mathrm{sc}}}(\mathbb{H}^n)$, where $q_{\mathrm{sc}}:=Q(p-1)/2(2-\gamma)$ is the scaling-related exponent, then the solution $u$ exists
globally in time, provided that $\|u_0\|_{L^{q_{\mathrm{sc}}}}$ is sufficiently small.
\end{enumerate}
\end{theorem}
\begin{remark} ${}$
\begin{itemize}
\item [$(\mathrm{a})$] Based on \eqref{esti1} and \eqref{esti2}, $p_c$ is the critical exponent of problem \eqref{1} when $0<\gamma<1$. In addition, as $$p_\gamma>p_{\mathrm{sc}}:=1+2(2-\gamma)/Q,$$ it follows that $p_c>p_{\mathrm{sc}}$, i.e. the critical exponent $p_c$ is not the one predicted by scaling. Furthermore, we note that $p>p_c\Rightarrow
q_{\mathrm{sc}}>1$ since $p_c>1+2(2-\gamma)/Q$.
\item[$(\mathrm{b})$] It is worth noting that in Theorem \ref{Blowup-global} one can actually replace $|u|^{p-1}u$ with a more general function $f(u)$ satisfying $$c|u|^{p}\leq f(u)\leq C|u|^{p},\qquad 0<c<C.$$ For example $f(u)=\left(1+u^{-2}\right)|u|^p$.
\item [$(\mathrm{c})$]  Observe that one may assume the decay condition $|u_0(\eta)|\leq C
(1+|\eta|_{_\mathbb{H}})^{-\kappa}$ with $\kappa>2(2-\gamma)/(p-1)$ instead of the integrability condition $u_0\in L^{q_{\mathrm{sc}}}(\mathbb{H}^n)$. This example highlights the relationship between the decay rate of $u_0$ and that described in Theorem \ref{Lifespan-2} below.
\end{itemize}
\end{remark} 
\begin{remark} From \eqref{int3A}, and Lemma \ref{lemma2} $\mathrm{(i)}$ below, it follows that for any $p >1$, there exists a constant $B> 0$ such that
	if $u_0\geq0$ satisfies
	$$\int_{\mathbb{H}^n}u_0(\eta)e^{-\varepsilon\sqrt{1+(|x|^2+|y|^2)^2+\tau^2}}\,d\eta\geq B,$$
	with  $\varepsilon=1/(2Q+4)$, then the mild solution of \eqref{1} blows up in finite time.
\end{remark} 
Next, we investigate estimates for the lifespan of local-in-time weak or mild solutions to equation \eqref{1}. Specifically, we aim to understand how the maximal existence time depends on the size of the initial data. To capture this dependency, we consider scaled initial data of the form \( u_0^\varepsilon = \varepsilon u_0 \), where \( \varepsilon > 0 \) is a small parameter and $u_0\in C_0(\mathbb{H})\cap L^1(\mathbb{H})$.

We define the lifespan \( T_\varepsilon \) as the supremum of all \( T > 0 \) such that the corresponding solution \( u^\varepsilon(t, \eta) \) exists and remains in the appropriate function space on the time interval \( [0, T) \). Our goal is to derive asymptotic estimates on \( T_\varepsilon \) as \( \varepsilon \to 0^+ \), thereby revealing how the smallness of the initial data governs the time interval of existence.

\begin{theorem}[Lifespan estimates]\label{Lifespan}${}$\\
Let $0\leq\gamma<1$, $p>1$, and $u_0\in C_0(\mathbb{H}^n)\cap L^1(\mathbb{H})$ such that $\displaystyle \int_{\mathbb{H}^n} u_0(\eta)\,d\eta>0$. If $$p< p_{\mathrm{sc}}=1+\frac{2(2-\gamma)}{Q},$$ then there exists $\varepsilon_0>0$ such that for any $\varepsilon\in(0,\varepsilon_0]$ it holds
$$T_\varepsilon\leq C \varepsilon^{-\left(\frac{2-\gamma}{p-1}-\frac{Q}{2}\right)^{-1}},$$
where $C>0$ is a positive constant independent of $\varepsilon$.
\end{theorem}
The result on the lifespan can be relaxed in the following manner.
\begin{theorem}[Lifespan estimates]\label{Lifespan-2}${}$\\
Let $0\leq\gamma<1$, $p>1$, and $u_0\in C_0(\mathbb{H}^n)\cap L^1_{\mathrm{loc}}(\mathbb{H})$ such that 
\begin{equation}\label{corA}
u_0(\eta)\geq (1+|\eta|)^{-\kappa},\qquad\hbox{for all}\,\,\eta\in\mathbb{H}^n.
\end{equation}
 If $0<\kappa<2(2-\gamma)/(p-1)$,  there exists a constant $\varepsilon_0>0$ such that the lifespan $T_\varepsilon$ verifies
$$
T_\varepsilon \leq\,C \, \varepsilon^{-\left(\frac{2-\gamma}{p-1}-\frac{\kappa}{2}\right)^{-1}}, \qquad \hbox{for any}\,\, \varepsilon \in ( 0,\varepsilon_0].
$$
\end{theorem}
\begin{remark} 
Unfortunately, Theorem \ref{Lifespan} does not provide a lifespan estimate for all cases of blow-up solutions. Since $p_{sc}<p_c$, we do not know the lifespan estimate for the case $p_{sc}\leq p\leq p_c.$ 
\end{remark}

The proofs of local and global existence of solutions rely on $L^p-L^q$ estimates and the application of a fixed point theorem. In contrast, the analysis of blow-up and lifespan of solutions is carried out using nonlinear capacity estimates, also known as the rescaled test function method. The nonlinear capacity method was introduced to prove the non-existence of global solutions by  Baras and Pierre \cite{Baras-Pierre}, then used by Baras and Kersner in \cite{Baras}; later on, it was developed by Zhang in \cite{Zhang} and Mitidieri and Pohozaev in \cite{19}.  It was also used by Kirane et al.  in \cite{Kirane-Guedda,KLT},  and Fino et al in \cite{Fino1,Fino3,Fino4,Fino5}. \\

The paper is organized as follows. Section \ref{sec2} introduces key definitions, terminology, and preliminary results related to the Heisenberg group, the sub-Laplacian operator, and its heat kernel. In Section \ref{sec3}, we derive the local existence of the mild solution.
Sections \ref{sec4}, \ref{sec5} and \ref{sec6} are dedicated to proving the main results of the paper.

\medskip
\noindent\textbf{Notation:}  We shall employ the following notational conventions in what follows. For $\Omega \subset \mathbb{R}$, the characteristic function on $\Omega$ is denoted by $\mathbbm{1}_\Omega$; for any nonnegative functions $g$ and $h$, we write $f\lesssim g$ whenever there exists a positive constant $C>0$ such that $f\leqslant Cg$. We use $C$ as a generic constant, which may change value from line to line. Norms in Lebesgue spaces $L^p$, for $p \in [1, \infty]$, are written as $\| \cdot \|_p$.


\section{Preliminaries}\label{sec2}
As a preliminary step, we recall the definition and fundamental properties of the Heisenberg group (cf. (see, e.g., \cite{Folland,Follandstein}).

\subsection{Heisenberg group}
The Heisenberg group $\mathbb{H}^n$  is the space $\mathbb{R}^{2n+1}=\mathbb{R}^n\times \mathbb{R}^n\times \mathbb{R}$ equipped with the group operation
$$\eta\circ\eta^\prime=(x+x^\prime,y+y^\prime,\tau+\tau^\prime+2(x\cdotp y^\prime-x^\prime\cdotp y)),$$
where $\eta=(x,y,\tau)$, $\eta^\prime=(x^\prime,y^\prime,\tau^\prime)$, and $\cdotp$ is the standard scalar product in $\mathbb{R}^n$. Let us denote the parabolic dilation in $\mathbb{R}^{2n+1}$ by $\delta_\lambda$, namely, $\delta_\lambda(\eta)=(\lambda x, \lambda y, \lambda^2 \tau)$ for any $\lambda>0$, $\eta=(x,y,\tau)\in\mathbb{H}^n$. The Jacobian determinant of $\delta_\lambda$ is $\lambda^Q$, where $Q=2n+2$ is the homogeneous dimension of $\mathbb{H}^n$. A direct calculation shows that $\delta_\lambda$ is an automorphism of $\mathbb{H}^n$ for every $\lambda>0$, and therefore  $\mathbb{H}^n=(\mathbb{R}^{2n+1},\circ,\delta_\lambda)$ is a homogeneous Lie group on $\mathbb{R}^{2n+1}$.

The homogeneous Heisenberg norm (also called Kor\'anyi norm) is derived from an anisotropic dilation on the Heisenberg group and is defined by
$$|\eta|_{_{\mathbb{H}}}=\left(\left(\sum_{i=1}^n (x_i^2+y_i^2)\right)^2+\tau^2\right)^{\frac{1}{4}}=\left((|x|^2+|y|^2)^2+\tau^2\right)^{\frac{1}{4}},$$
where $|\cdotp|$ is the Euclidean norm associated with $\mathbb{R}^n$. The associated Kor\'anyi distance between two points $\eta$ and $\xi$ of $\mathbb{H}$ is defined by
$$d_{_{\mathbb{H}}}(\eta,\xi)=|\xi^{-1}\circ\eta|_{_{\mathbb{H}}},\quad \eta,\xi\in\mathbb{H},$$
where $\xi^{-1}$ denotes the inverse of $\xi$ with respect to the group action, i.e. $\xi^{-1}=-\xi$. This metric induces a topology on $\mathbb{H}^n$. Thus, we can define the Heisenberg ball of $\mathbb{H}^n$, centered
at $\eta$ and with radius $r>0$, as $$B_{\mathbb{H}}(\eta,r)=\{\xi\in\mathbb{H}^n:\,d_{_{\mathbb{H}}}(\eta,\xi)<r\}.$$

The Heisenberg convolution between any two regular functions $f$ and $g$ is defined by
\begin{align*}
(f\ast_{_{\mathbb{H}}}g)(\eta)&=\int_{\mathbb{H}^n}f(\eta\circ\xi^{-1})g(\xi)\,d\xi\\&=\int_{\mathbb{H}^n}f(\xi)g(\xi^{-1}\circ\eta)\,d\xi.\end{align*}

The left-invariant vector fields that span the Lie algebra \(\mathfrak{h}_n\) of \(\mathbb{H}^n\) are given by
\[
X_i = \partial_{x_i}-2 y_i\partial_\tau, \quad
Y_i = \partial_{y_i}+2 x_i\partial_\tau, \quad
T = \partial_\tau, \quad i = 1, \ldots, n.
\]
The Heisenberg gradient is given by
\begin{equation}\label{48}
\nabla_{\mathbb{H}}=(X_1,\dots,X_n,Y_1,\dots,Y_n),
\end{equation}
and the sub-Laplacian (also referred to as Kohn Laplacian) is defined as
\begin{equation}\label{40}
\Delta_{\mathbb{H}}=\sum_{i=1}^n(X_i^2+Y_i^2)=\Delta_x+\Delta_y+4(|x|^2+|y|^2)\partial_\tau^2+4\sum_{i=1}^{n}\left(x_i\partial_{y_i\tau}^2-y_i\partial_{x_i\tau}^2\right),
\end{equation}
where $\Delta_x=\nabla_x\cdotp\nabla_x$ and $\Delta_y=\nabla_y\cdotp\nabla_y$ stand for the Laplace operators on $\mathbb{R}^n$.

\subsection{Heat kernel}
We recall the definition and some properties related to the heat kernel associated with $-\Delta_{\mathbb{H}}$ on the Heisenberg group.
\begin{proposition}\label{Heat}\cite[Theorem~1.8]{BahouriGallagher}${}$\\
There exists a function $h\in\mathcal{S}(\mathbb{H}^n)$ such that if u denotes the
solution of the free heat equation on the Heisenberg group
\begin{equation}\label{}
\begin{array}{ll}
\partial_t u-\Delta_{\mathbb{H}}u=0,&\qquad {\eta\in \mathbb{H}^n,\,\,\,t>0,}
 \\{}\\ u(\eta,0)=u_{0}(\eta), &\qquad \eta\in \mathbb{H}^n,
 \end{array}
\end{equation}
then we have
$$u(\cdotp,t)=h_t\ast_{_{\mathbb{H}}} u_0,$$
where $\ast_{_{\mathbb{H}}}$ denotes the convolution on the Heisenberg group defined above, while the heat kernel $h_t$ associated with $-\Delta_{\mathbb{H}}$ is defined by
$$h_t(\eta)=\frac{1}{t^{n+1}}h\left(\frac{x}{\sqrt{t}},\frac{y}{\sqrt{t}},\frac{s}{t} \right),\qquad\hbox{for all}\,\,\,\eta=(x,y,s)\in\mathbb{H}^n, \,\,t>0.$$
\end{proposition} 
The following proposition follows directly from \cite[Theorem~3.1]{Folland}, \cite[Corollary~1.70]{Follandstein}, \cite[Proposition~1.68]{Follandstein}, and \cite[Corollary~2.3]{Pazy}.
\begin{proposition}\label{properties}${}$\\
There is a unique semigroup $(S_{\mathbb{H}}(t))_{t>0}$ generated by $\Delta_{\mathbb{H}}$ and satisfying the following properties:
\begin{itemize}
\item[$(i)$] $(S_{\mathbb{H}}(t))_{t>0}$ is a contraction semigroup on $L^p(\mathbb{H}^n)$, $1\leq p\leq \infty$, which is strongly continuous for $p<\infty$;
\item[$(ii)$] $(S_{\mathbb{H}}(t))_{t>0}$  is a strongly continuous semigroup on $C_0(\mathbb{H}^n)$;
\item[$(iii)$] For every $v\in X$, where $X$ is either $L^p(\mathbb{H}^n)$ for $1\leq p< \infty$ or $C_0(\mathbb{H}^n)$, the map $t\longmapsto S_{\mathbb{H}}(t)v$ is continuous from $[0,\infty)$ into $X$;
\item[$(iv)$] $S_{\mathbb{H}}(t)f=h_t\ast_{_{\mathbb{H}}} f$, for all $f\in L^p(\mathbb{H}^n)$, $1\leq p\leq \infty$, $t>0$;
\item[$(v)$] $\displaystyle\int_{\mathbb{H}^n}h_t(\eta)\,d\eta=1$, for all $t>0$;
\item[$(vi)$] $h_t(\eta)\geq0$, for all $\eta \in \mathbb{H}^n, t>0$; 
\item[$(vii)$] $h_{r^2t}(rx,ry,r^2s)=r^{-Q}h_{t}(x,y,s)$, for all $r,t>0$, $(x, y,s) \in \mathbb{H}^n$; 
\item[$(viii)$] $h_t(\eta)=h_t(\eta^{-1})$, for all $\eta \in \mathbb{H}^n$, $t>0$;
\item[$(ix)$] $h_t\ast_{_{\mathbb{H}}} h_s=h_{t+s}$, for all $s,t>0$. 
\end{itemize}
\end{proposition} 
\begin{proposition}\cite[Theorem~1.3.2]{Randall}${}$\\
The function \( h_t \) is given by
\[
h_t(\eta) = \frac{1}{(2\pi)^{n+2} 2^n} \int_{\mathbb{R}} \left(\frac{\lambda}{\sinh(t\lambda)}\right)^n
\exp\left( -\frac{|z|^2 \lambda}{4 \tanh(t\lambda)} \right) 
e^{i\lambda \tau} \, d\lambda, \quad \hbox{for all}\,\,\eta=(z, \tau) \in \mathbb{R}^{2n} \times \mathbb{R}.
\]
\end{proposition} 
\begin{proposition}\label{pointwiseupperbound} \cite[Theorems~2,4]{Jerison} or \cite[Theorems~IV.4.2-4.3]{Varapoulos}${}$\\
Let $h_t$ be the heat kernel associated with $-\Delta_{\mathbb{H}}$. Then there exist two positive constants $c_\star< C_\star$ such that the estimate
\[
c_\star t^{-\frac{Q}{2}} \exp\left( -\frac{C_\star |\eta|_H^2}{t} \right) 
\leq h_t(\eta) 
\leq C_\star t^{-\frac{Q}{2}} \exp\left( -\frac{c_\star |\eta|_H^2}{t} \right) 
\]
holds true, for any $t > 0$ and $\eta \in \mathbb{H}^n$.
\end{proposition} 
\begin{proposition}\label{pointwiseheatestimation}\cite[Theorem~1]{Jerison} or \cite[Theorem~IV.4.2]{Varapoulos}${}$\\
Let $h_t$ be the heat kernel associated with $-\Delta_{\mathbb{H}}$. Then there exist positive constants $c_1$ and $C_{I,l}$ depending on $-\Delta_{\mathbb{H}}$ such that
$$|\partial_t^lX_I h_t(\eta)|\leq C_{I,l}t^{-l-\frac{|I|}{2}-\frac{Q}{2}}e^{-\frac{c_1|\eta|^2_{_{\mathbb{H}}}}{t}},\qquad\hbox{for all}\,\,\,\eta\in\mathbb{H}^n,\,\,t>0,$$
where $I=(i_1,\dots,i_m)$ with $|I|=m$ and $X_I=X_{i_1}X_{i_2}\dots X_{i_m}$.
\end{proposition} 
The following result is a direct consequence of Proposition \ref{pointwiseheatestimation}.
\begin{lemma}\label{estimateheatkernel}
Let $h_t$ be the heat kernel associated with $-\Delta_{\mathbb{H}}$. Then there exist positive constants $C_1$ and $C_2$ depending on $-\Delta_{\mathbb{H}}$ such that
$$\left\|\nabla_{\mathbb{H}} h_t(\cdotp)\right\|_1\leq C_1 t^{-\frac{1}{2}},\qquad \left\|T h_t(\cdotp)\right\|_1\leq C_2 t^{-\frac{1}{2}},\qquad t>0,$$
where $\nabla_{\mathbb{H}}$ is defined by \eqref{48} and $T=\partial_\tau$.
\end{lemma}
\begin{lemma}(Young's inequality)\label{Younginequality}\cite[Proposition~1.8]{Follandstein}${}$\\
Suppose $1\leq p,q,r\leq \infty$ and $\frac{1}{p}+\frac{1}{q}=1+\frac{1}{r}$. If $f\in L^p(\mathbb{H}^n)$ and $g\in L^q(\mathbb{H}^n)$, then $f\ast_{_{\mathbb{H}}}g\in L^r(\mathbb{H}^n)$ and the following inequality holds
$$\|f\ast_{_{\mathbb{H}}}g\|_r\leq \|f\|_p\|g\|_q.$$
\end{lemma}
Thus, by applying Young's inequality (Lemma \ref{Younginequality}) and using Proposition \ref{pointwiseheatestimation} along with the homogeneity property 
$$|(\lambda x,\lambda y,\lambda^2 \tau)|_{_{\mathbb{H}}}=\lambda |\eta|_{_{\mathbb{H}}},\quad \hbox{for}\,\, \eta=(x,y,\tau)\in\mathbb{H}^n,\,\,\lambda>0,$$ 
we obtain the following estimate.
\begin{lemma}\label{Lp-Lqestimate}($L^p-L^q$ estimate)${}$\\
 Suppose $1\leq p\leq q\leq \infty$. Then there exists a positive constant $C$ such that for every $f\in L^p(\mathbb{H}^n)$, the following inequality holds:
$$\|S_{\mathbb{H}}(t)f\|_q\leq C\,t^{-\frac{Q}{2}(\frac{1}{p}-\frac{1}{q})}\|f\|_p,\qquad t>0.$$
\end{lemma}
\subsection{Test functions}
In preparation for later use, we introduce and describe  the properties of the cut-off functions $\{\varphi_R\}_{R>0}$
defined as 
\begin{equation}\label{varphi_R}
\varphi_R(\eta)= \Phi^{\ell}\left(\xi_R(\eta)\right),\quad \xi_R(\eta)=\frac{|\eta|_{_{\mathbb{H}}}^2}{R}\quad\eta=(x,y,\tau)\in\mathbb{H}^n,
\end{equation}
where $\ell=2p'$, $p'=p/(p-1)$ is the H\"older conjugate of $p$, and $\Phi\in \mathcal{C}^\infty(\mathbb{R})$ is a smooth non-increasing function 
satisfying $\mathbbm{1}_{(-\infty,\frac{1}{2}]} \leq \Phi\leq \mathbbm{1}_{(-\infty,1]}$.  Furthermore, we define 
$$\varphi_R^*(\eta)=\Phi^\ell_*\left(\xi_R(\eta)\right),\quad\hbox{where}\quad \Phi_*=\mathbbm{1}_{[\frac{1}{2},1]}\Phi.$$
\begin{lemma}\label{lem:testfunction}
Define the family of cut-off functions $\{\varphi_R\}_{R>0}$
as in \eqref{varphi_R}.
Then the following inequality holds:
\[
|\Delta_{\mathbb{H}}\varphi_R(\eta)|\leq \frac{C}{R}
\left(\varphi_R^*(\eta)\right)^{\frac{1}{p}}.
\]
\end{lemma}
\begin{proof}
It is easy to see that 
\begin{eqnarray*}
\left|\Delta_{\mathbb{H}}\varphi_R(\eta)\right|&\leq&\left|\Delta_x\Phi^\ell\big(\xi_R(\eta)\big)\right|+\,\left|\Delta_y\Phi^{\ell}\big(\xi_R(\eta)\big)\right|+\,4(|x|^2+|y|^2)\left|\partial_\tau^2\Phi^{\ell}\big(\xi_R(\eta)\big)\right|\\
&{}&+\,4\sum_{j=1}^{n}|x_j|\left|\partial^2_{y_j\tau}\Phi^{\ell}\big(\xi_R(\eta)\big)\right| +\,4\sum_{j=1}^{n}|y_j|\left|\partial^2_{x_j\tau}\Phi^{\ell}\big(\xi_R(\eta)\big)\right|.
\end{eqnarray*}
So,
\begin{eqnarray*}
\left|\Delta_{\mathbb{H}}\varphi_R(\eta)\right|&\lesssim&\Phi^{\ell-2}_*\big(\xi_R(\eta)\big)\left(|\Phi'_*(\xi_R(\eta))|^2+|\Phi''_*(\xi_R(\eta))|\right)\left(\left|\nabla_x\xi_R(\eta)\right|^2+\left|\nabla_y\xi_R(\eta)\right|^2\right)\\
&{}&+\,\Phi^{\ell-1}_*\big(\xi_R(\eta)\big)|\Phi'_*(\xi_R(\eta))|\left(\left|\Delta_x\xi_R(\eta)\right|+\left|\Delta_y\xi_R(\eta)\right|\right)\\
&{}&+\,(|x|^2+|y|^2)\Phi^{\ell-2}_*\big(\xi_R(\eta)\big)\left(|\Phi'_*(\xi_R(\eta))|^2+|\Phi''_*(\xi_R(\eta))|\right)\left|\partial_\tau\xi_R(\eta)\right|^2\\
&{}&+\,(|x|^2+|y|^2)\Phi^{\ell-1}_*\big(\xi_R(\eta)\big)|\Phi'_*(\xi_R(\eta))|\left|\partial^2_\tau\xi_R(\eta)\right|\\
&{}&+\,\Phi^{\ell-2}_*\big(\xi_R(\eta)\big)\left(|\Phi'_*(\xi_R(\eta))|^2+|\Phi''_*(\xi_R(\eta))|\right)\left|\partial_\tau\xi_R(\eta)\right|\sum_{j=1}^{n}|x_j|\left|\partial_{y_j}\xi_R(\eta)\right|\\
&{}&+\,\Phi^{\ell-1}_*\big(\xi_R(\eta)\big)|\Phi'_*(\xi_R(\eta))|\sum_{j=1}^{n}|x_j|\left|\partial^2_{y_j\tau}\xi_R(\eta)\right|\\
&{}&+\,\Phi^{\ell-2}_*\big(\xi_R(\eta)\big)\left(|\Phi'_*(\xi_R(\eta))|^2+|\Phi''_*(\xi_R(\eta))|\right)\left|\partial_\tau\xi_R(\eta)\right|\sum_{j=1}^{n}|y_j|\left|\partial_{x_j}\xi_R(\eta)\right|\\
&{}&+\,\Phi^{\ell-1}_*\big(\xi_R(\eta)\big)|\Phi'_*(\xi_R(\eta))|\sum_{j=1}^{n}|y_j|\left|\partial^2_{x_j\tau}\xi_R(\eta)\right|.
\end{eqnarray*}
For any $\eta=(x,y,\tau)\in\mathbb{H}^n$, we introduce the rescaled variable $\widetilde{\eta}=(\widetilde{x},\widetilde{y},\widetilde{\tau})\in\mathbb{H}^n$ where
$$\widetilde{x}=\frac{x}{\sqrt{R}},\qquad\widetilde{y}=\frac{y}{\sqrt{R}},\qquad\widetilde{\tau}=\frac{\tau}{R},$$
which yields
\begin{eqnarray*}
\left|\Delta_{\mathbb{H}}\varphi_R(\eta)\right|&\lesssim&R^{-1}\Phi^{\ell-2}_*\big(|\widetilde{\eta}|^2\big)\left(|\Phi'_*(|\widetilde{\eta}|^2)|^2+|\Phi''_*(|\widetilde{\eta}|^2)|\right)+R^{-1}\,\Phi^{\ell-1}_*\big(|\widetilde{\eta}|^2\big)|\Phi'_*(|\widetilde{\eta}|^2)|.
\end{eqnarray*}
By taking into account the fact that $\Phi_*\leq1$, $\Phi_*\in C^\infty$ and the support of $\Phi_*$, a direct computation shows that
$$\left|\Delta_{\mathbb{H}}\varphi_R(\eta)\right|\lesssim R^{-1} \left(\varphi_R^*(\eta)\right)^{\frac{1}{p}}.$$
\end{proof}
\subsection{Fractional Calculus}
\begin{definition}
A function $\mathcal{A} :[a,b]\rightarrow\mathbb{R}$, $-\infty<a<b<\infty$, is said to be absolutely continuous if and only if there exists $\psi\in L^1(a,b)$ such that
\[
\mathcal{A} (t)=\mathcal{A} (a)+\int_{a}^t \psi(s)\,ds.
\]
$AC[a,b]$ denotes the space of these functions.  Moreover, 
\[
AC^2[a,b]:=\left\{\varphi:[a,b]\rightarrow\mathbb{R}\,\,\hbox{such that}\,\,\varphi'\in
AC[a,b]\right\}.
\]
\end{definition}

\begin{definition}
Let $f\in L^1(c,d)$, $-\infty<c<d<\infty$. The left- and right-sided Riemann-Liouville fractional integrals are defined by
\begin{equation}\label{I1}
I^\alpha_{c|t}f(t):=\frac{1}{\Gamma(\alpha)}\int_{c}^t(t-s)^{-(1-\alpha)}f(s)\,ds, \quad t>c, \; \alpha\in(0,1),
\end{equation}
and
\begin{equation}\label{I2}
I^\alpha_{t|d}f(t):=\frac{1}{\Gamma(\alpha)}\int_t^{d}(s-t)^{-(1-\alpha)}f(s)\,ds, \quad t<d, \; \alpha\in(0,1)
\end{equation}
where $\Gamma$ is the Euler gamma function.
\end{definition}

\begin{definition}
Let $f\in AC[c,d]$, $-\infty<c<d<\infty$. The left- and right-sided Riemann-Liouville fractional derivatives  are defined by
\begin{equation}\label{}
D^\alpha_{c|t}f(t):=\frac{d}{dt}I^{1-\alpha}_{c|t}f(t)=\frac{1}{\Gamma(1-\alpha)}\frac{d}{dt}\int_{c}^t(t-s)^{-\alpha}f(s)\,ds, \quad t>c,  \; \alpha\in(0,1),
\end{equation}
and
\begin{equation}\label{}
D^\alpha_{t|d}f(t):=-\frac{d}{dt}I^{1-\alpha}_{t|d}f(t)=-\frac{1}{\Gamma(1-\alpha)}\frac{d}{dt}\int_t^{d}(s-t)^{-\alpha}f(s)\,ds, \quad t<d, \; \alpha\in(0,1).
\end{equation}
\end{definition}

\begin{proposition}\cite[(2.64), p. 46]{SKM}\\
Let $\alpha\in(0,1)$ and $-\infty<c<d<\infty$. The fractional integration by parts formula
\begin{equation}\label{IP}
\int_{c}^{d}f(t)D^\alpha_{c|t}g(t)\,dt \;=\; \int_{c}^{d}
g(t)D^\alpha_{t|d}f(t)\,dt
\end{equation}
is satisfied for every $f\in I^\alpha_{t|d}(L^p(c,d))$, $g\in I^\alpha_{c|t}(L^q(c,d))$ such that $\frac{1}{p}+\frac{1}{q}\leq 1+\alpha$, $p,q>1$, where
\[
I^\alpha_{c|t}(L^q(0,T)):=\left\{f= I^\alpha_{c|t}h,\,\, h\in L^q(c,d)\right\},
\]
and
\[
I^\alpha_{t|d}(L^p(c,d)):=\left\{f= I^\alpha_{t|d}h,\,\, h\in L^p(c,d)\right\}.
\]
\end{proposition}
\begin{remark}
A simple sufficient condition for functions $f$ and $g$ to satisfy (\ref{IP}) is that $f,g\in C[c,d],$ such that
$D^\alpha_{t|d}f(t),D^\alpha_{c|t}g(t)$ are continuous on $[c,d]$.
\end{remark}

\begin{proposition}\cite[Chapter~1]{SKM}\\
For $0<\alpha<1$, $-\infty<c<d<\infty$, we have the following identities
\begin{equation}\label{I3}
    D^\alpha_{c|t}I^\alpha_{c|t}\varphi(t)=\varphi(t),\,\hbox{a.e. $t\in(c,d)$}, \quad\hbox{for all}\,\varphi\in L^r(c,d), 1\leq r\leq\infty,
\end{equation}
and
\begin{equation}\label{I4}
   -D D^\alpha_{t|d}\varphi=D^{1+\alpha}_{t|d}\varphi,\quad \varphi\in AC^2[c,d],
\end{equation}
where $D:=\frac{d}{dt}$.
\end{proposition}

Given $T>0$, let us define the function $w_1$ by
\begin{equation}\label{w1}
\displaystyle w_1(t)=\left(1-t/T\right)^\sigma,\quad  t \in [0, T], \; \sigma\gg1.
\end{equation}
The following properties of $w_1$ will be used in the sequel.
\begin{lemma}\cite[(2.45), p. 40]{SKM}\label{L1}\\
Let $T>0$, $0<\alpha<1$. For all $t\in[0,T]$, we have
\begin{equation}\label{P1}
D_{t|T}^\alpha
w_1(t)=\frac{\Gamma(\sigma+1)}{
\Gamma(\sigma+1-\alpha)}T^{-\alpha}(1-t/T)^{\sigma-\alpha},
\end{equation}
and
\begin{equation}\label{P3}
D_{t|T}^{1+\alpha}
w_1(t)=\frac{\Gamma(\sigma+1)}{
\Gamma(\sigma-\alpha)}T^{-(1+\alpha)}(1-t/T)^{\sigma-\alpha-1}.
\end{equation}
\end{lemma}

\begin{lemma}\label{lemma1}\cite[Lemma~2.6]{FRT}${}$\\
For $\varepsilon>0$, $A>0$, $c>0$, let
$$\Theta(\eta)=ce^{-\varepsilon\sqrt{A+(|x|^2+|y|^2)^2+\tau^2}},\qquad \eta=(x,y,\tau)\in \mathbb{H}^n.$$
Then
$$
\Delta_{\mathbb{H}}\Theta(\eta)\geq -2\varepsilon(Q+2)\Theta(\eta),\qquad\hbox{for all}\,\,\eta\in\mathbb{H}^n.
$$
\end{lemma}
\begin{lemma}\cite[Proposition~2.2]{CDW}\label{lemma2} \\
	Let $T>0$, $0\leq \gamma<1$, $p>1$, $a>0$, and $b>0$, there exists a constant $K=K(a,b,p)$ such that, if $w\in C^1([0,T])$, $w>0$, and satisfies 
	$$w'(t)+a\,w(t)\geq b\int_0^t(t-s)^{-\gamma}w^p(s)\,ds,\quad\hbox{for all}\,\,T>0,$$
	then the following properties hold:
	\begin{itemize}
		\item[$\mathrm{(i)}$] If $T\geq 1$, then $w(0)<K$.
		\item[$\mathrm{(ii)}$] If $p\gamma\leq 1$, then $T<\infty$.
	\end{itemize}
\end{lemma}

\section{Local existence}\label{sec3}
\begin{proof}[Proof of Theorem \ref{Local}]
	 For arbitrary $T>0,$ we define the Banach space
\[E_T:=\left\{u\in L^\infty((0,T),C_0(\mathbb{H}^n));\;
\vertiii{u}\leq 2\|u_0\|_{L^\infty}\right\},
\]
where
$\vertiii{u}:=\|\cdotp\|_{L^\infty((0,T),L^\infty(\mathbb{H}^n))}.$ Next, for
every $u\in E_T,$ we define
\[
\Phi(u)(t,\eta):=S_{\mathbb{H}}(t) u_0(\eta)+\Gamma(\alpha)\int_{0}^tS_{\mathbb{H}}(t-s)I_{0|s}^\alpha(|u|^{p-1}u)(s,\eta)\,ds.
\]
We prove the local existence by the Banach fixed point
theorem.\\
\noindent {\it Step 1. ${\bf\Phi:E_T\rightarrow E_T}$.} Let $u\in E_T,$ using Lemma \ref{Lp-Lqestimate}, we obtain
\begin{eqnarray*}
\vertiii{\Phi(u)} &\leq& \|u_0\|_\infty+\left\|
  \int_0^t\int_0^s(s-\sigma)^{-\gamma}\|u(\sigma)\|^p_\infty\,d\sigma\,ds\right\|_{L^\infty(0,T)}\\
   &=& \|u_0\|_\infty+\left\|
   \int_0^t\int_\sigma^t(s-\sigma)^{-\gamma}\|u(\sigma)\|^p_\infty\,ds\,d\sigma\right\|_{L^\infty(0,T)} \\
   &\leq&\|u_0\|_\infty+\frac{T^{2-\gamma}}{(1-\gamma)(2-\gamma)}\vertiii{u}^p\\
   &\leq&\|u_0\|_\infty+\frac{T^{2-\gamma}2^p\|u_0\|_{L^\infty}^{p-1}}
   {(1-\gamma)(2-\gamma)}\|u_0\|_\infty.
\end{eqnarray*}
Now, if we choose $T$ small enough such that
\begin{equation}\label{conditionssurT+}
\frac{T^{2-\gamma}2^p\|u_0\|_{\infty}^{p-1}}{(1-\gamma)(2-\gamma)}\leq1,
\end{equation}
we conclude that $\vertiii{\Phi(u)}\leq 2\|u_0\|_{\infty},$ and then $\Phi(u)\in E_T.$\\

\noindent {\it Step 2. $\Phi$ is a contraction.}  For $u,v\in E_T$, taking account of Lemma \ref{Lp-Lqestimate}, we have
\begin{eqnarray*}
 \vertiii{\Psi(u)-\Psi(v)}&\leq&\left\|
  \int_0^t\int_0^s(s-\sigma)^{-\gamma}
  \||u|^{p-1}u(\sigma)-|v|^{p-1}v(\sigma)\|_{\infty}\,d\sigma\,ds\right\|_{L^\infty(0,T)}\\
   &=&\left\|
   \int_0^t\int_\sigma^t(s-\sigma)^{-\gamma}\||u|^{p-1}u(\sigma)-
   |v|^{p-1}v(\sigma)\|_{\infty}\,ds\,d\sigma\right\|_{L^\infty(0,T)} \\
   &\leq&\frac{T^{2-\gamma}}{(1-\gamma)(2-\gamma)} \vertiii{|u|^{p-1}u-|v|^{p-1}v}  \\
   &\leq& \frac{C(p)2^p\|u_0\|_{\infty}^{p-1}T^{2-\gamma}}
   {(1-\gamma)(2-\gamma)} \vertiii{u-v} \\
   &\leq& \frac{1}{2} \vertiii{u-v},
\end{eqnarray*}
thanks to the following inequality
\begin{equation}\label{estimationimp}
||u|^{p-1}u-|v|^{p-1}v|\leq C(p)|u-v|(|u|^{p-1}+|v|^{p-1});
\end{equation}
$T$ is chosen such that
\begin{equation}\label{esti5+}
\frac{T^{2-\gamma}2^p\|u_0\|_{\infty}^{p-1}\max(2C(p),1)}{(1-\gamma)(2-\gamma)}\leq1.
\end{equation}
Then, by the Banach fixed point
theorem, there exists a unique mild solution $u\in E_T$ to the problem \eqref{1}.\\

\noindent {\it Step 3. Continuity.} One can easily check that $f\in L^1((0,T),C_0(\mathbb{H}^{n}))$, with
$$f(t):= \int_0^t(t-s)^{-\gamma}|u(s)|^{p-1}u(s) \, \mathrm{d}s,\qquad \hbox{for all}\,\,t\in (0,T).$$
 Applying Lemma 4.1.5 in \cite{CH}, using the continuity of the semigroup $S_{\mathbb{H}}(t)$,  we conclude that $u\in C([0,T],C_0(\mathbb{H}^{n}))$.\\
 
\noindent {\it Step 4. Uniqueness in $\Pi_T:=C([0,T],C_0(\mathbb{H}^n))$.} If $u,v$ are two mild solutions in $\Pi_T$ for some $T>0,$
using Lemma \ref{Lp-Lqestimate} and  \eqref{estimationimp}, we obtain
\begin{eqnarray*}
  \|u(t)-v(t)\|_{\infty}&\leq&C(p)2^p(\vertiii{u}^{p-1}+\vertiii{v}^{p-1})
  \int_0^t\int_0^s(s-\sigma)^{-\gamma}
  \|u(\sigma)-v(\sigma)\|_{\infty}\,d\sigma\,ds\\
   &=&C(p)2^p(\vertiii{u}^{p-1}+\vertiii{v}^{p-1})
   \int_0^t\int_\sigma^t(s-\sigma)^{-\gamma}\|u(\sigma)-
   v(\sigma)\|_{\infty}\,ds\,d\sigma \\
   &=&\frac{C(p)2^p(\vertiii{u}^{p-1}+\vertiii{v}^{p-1})}
   {1-\gamma} \int_0^t(t-\sigma)^{1-\gamma}\|u(\sigma)-
   v(\sigma)\|_{\infty}\,d\sigma\\
    &\leq&\frac{C(p)2^p(\vertiii{u}^{p-1}+\vertiii{v}^{p-1})T^{1-\gamma}}
   {1-\gamma} \int_0^t\|u(\sigma)-
   v(\sigma)\|_{\infty}\,d\sigma.
\end{eqnarray*}
So the uniqueness follows from Gronwall's inequality.

Using the uniqueness of solutions, we conclude the existence
of a unique maximal solution 
\[
u\in C([0,T_{\max}),C_0(\mathbb{H}^n)),
\]
where
\[
T_{\max}:=\sup\left\{T>0\;;\;\text{there exists a mild solution $u\in\Pi_T$
to \eqref{1}}\right\}\leq+\infty.
\]

\noindent {\it Step 5. Alternative.} We proceed by contradiction. If
$$\liminf_{t\rightarrow T_{\max}}\|u\|_{L^\infty((0,t),L^\infty(\mathbb{H}^{n}))}=:L<\infty,$$
then there exists a time sequence $\{t_m\}_{m\geq0}$ tending to $T_{\max}$ as $m\rightarrow\infty$ and such that
$$\sup_{m\in\mathbb{N}}\sup_{0<s\leq t_m}\|u(s)\|_{L^q}\leq L+1.$$
Moreover, if $0\leq t_m\leq t_m+\tau< T_{\max},$ using \eqref{IE}, we can write
\begin{eqnarray}\label{newIE+}
u(t_m+\tau)&=&S_{\mathbb{H}}(\tau)u(t_m)+\int_0^\tau S_{\mathbb{H}}(\tau-s)\int_0^s(s-\sigma)^{-\gamma}|u|^{p-1}u(t_m+\sigma)\,d\sigma\,ds\nonumber\\
&& +\int_0^\tau S_{\mathbb{H}}(\tau-s)\int_0^{t_m}(t_m+s-\sigma)^{-\gamma}|u|^{p-1}u(\sigma)\,d\sigma\,ds.
\end{eqnarray}
Using the fact that the last term on the right-hand side in \eqref{newIE+} depends only on the values of $u$ in the interval
$(0,t_m)$ and using again a fixed-point argument with $u(t_m)$ as initial condition, one can deduce that there exists $T(L + 1) > 0$, depends on $L+1$, such that the solution
$u(t)$ can be extended on the interval $[t_m, t_m + T(L + 1)]$ for any $m\geq0$. Thus, by the definition of the maximality time, $$T_{\max}\geq t_m+T(L+1),$$ for any $m\geq0$. We get the desired contradiction by letting $m\rightarrow\infty$.\\

\noindent {\it Step 6. Positivity of solutions.} If $u_0\geq0$ and $u_0\not\equiv0,$ then we can construct
a nonnegative solution on some interval $[0,T]$ by applying the
fixed point argument in the set $E_T^+=\{u\in E_T;\;u\geq0\}.$ In
particular, it follows from \eqref{IE} and Proposition \ref{pointwiseupperbound} that $$u(t)\geq S_{\mathbb{H}}(t)u_0>0\,\,\,\,\qquad
\text{on} \,\,\,\,(0,T].$$  It is not difficult by uniqueness to deduce that $u$ stays positive
on $(0,T_{\max}).$\\

\noindent {\it Step 7. Regularity.} If $u_0\in
L^r(\mathbb{H}^n)\cap C_0(\mathbb{H}^n),$ for $1\leq r<\infty,$ then by repeating the
fixed point argument in the space
\[
  E_{T,r}:= \{u\in L^\infty((0,T),C_0(\mathbb{H}^n)
  \cap L^r(\mathbb{H}^n));\;\vertiii{u}\leq 2\|u_0\|_{L^\infty},\|u\|_{\infty,r}\leq
   2\|u_0\|_{L^r}\},
\]
instead of $E_T,$ where $\|\cdotp\|_{\infty,r}:=\|\cdotp\|_{L^\infty((0,T),L^r(\mathbb{H}^n))},$
and by estimating $\|u^p\|_{L^r(\mathbb{H}^n)}$ by
$\|u\|^{p-1}_{L^\infty(\mathbb{H}^n)}\|u\|_{L^r(\mathbb{H}^n)}$ in
the contraction mapping argument, using \eqref{Lp-Lqestimate}, we obtain a
unique solution in $E_{T,r};$ we conclude then that
\[
u\in
C([0,T_{\max}),C_0(\mathbb{H}^n)\cap L^r(\mathbb{H}^n)).
\] The proof is complete.
\end{proof}

\section{Finite-time blow up}\label{sec4}
This section is devoted to proving the blow-up part of Theorem \ref{Blowup-global}. The idea of the proof is to use the variational formulation of the weak solution and choose an appropriate test function. We address the cases $p\gamma\leq 1$ and $p\leq p_\gamma$ separately.

\begin{proof}[The case $p\gamma\leq 1$]

	The proof is by contradiction. Suppose that $u$ is a global mild solution to \eqref{1}, then,  $u\geq0$ and
	\begin{equation}\label{IEE}
		u(t)=S_{\mathbb{H}}(t)u_0+\Gamma(\alpha)\int_0^tS_{\mathbb{H}}(t-s)I_{0|s}^\alpha
		(u^{p})(s)\,ds,\quad t\in[0,T],
	\end{equation}
	for all $T\gg1$, where $\alpha=1-\gamma$. Multiply \eqref{IEE} by $\Theta$, defined  in Lemma \ref{lemma1}, and integrating over $\mathbb{H}^n$, we have
	\[
	\int_{\mathbb{H}^n}u(t,\eta)\Theta(\eta)\,\mathrm{d\eta}= \int_{\mathbb{H}^n}S_{\mathbb{H}}(t)u_0(\eta)\Theta(\eta)\,\mathrm{d\eta}+\Gamma(\alpha)\int_{\mathbb{H}^n}\int_0^t S_{\mathbb{H}}(t-s)I_{0|s}^\alpha
	(u^p)(s,\eta)\Theta(\eta)\,ds\,\mathrm{d\eta}.
	\]
	We differentiate to obtain
	\begin{eqnarray}\label{differentiation}
		\frac{d}{\mathrm{dt}} \ \int_{\mathbb{H}^n}u(t,\eta)\Theta(\eta)\,\mathrm{d\eta}&=& \int_{\mathbb{H}^n} \frac{d}{\mathrm{dt}}\left(S_{\mathbb{H}}(t)u_0(\eta)\right)\Theta(\eta)\,\mathrm{d\eta}\\
		&&\,+\Gamma(\alpha)\int_{\mathbb{H}^n} \frac{d}{\mathrm{dt}}\int_0^t S_{\mathbb{H}}(t-s)I_{0|s}^\alpha(u^p)(s,\eta)\,ds\Theta(\eta)\,\mathrm{d\eta}.\nonumber
	\end{eqnarray}
	Using the properties of the semigroup $S(t)$, we have
	\begin{equation}\label{int1A}\begin{split}
		\int_{\mathbb{H}^n} \frac{d}{\mathrm{dt}}\left(S_{\mathbb{H}}(t)u_0(\eta)\right)\Theta(\eta)\,\mathrm{d\eta}&=\int_{\mathbb{H}^n} \Delta_{\mathbb{H}}\left(S_{\mathbb{H}}(t)u_0(\eta)\right)\Theta(\eta)\,\mathrm{d\eta}\\&=\int_{\mathbb{H}^n} S_{\mathbb{H}}(t)u_0(\eta)\Delta_{\mathbb{H}}\Theta(\eta)\,\mathrm{d\eta}.\end{split}
	\end{equation}
	Similarly,
	\begin{eqnarray}\label{int2A}
		&&\int_{\mathbb{H}^n} \frac{d}{\mathrm{dt}}\int_0^t S_{\mathbb{H}}(t-s)I_{0|s}^\alpha(u^p)(s,\eta)\,ds\Theta(\eta)\,\mathrm{d\eta}\nonumber\\
		&&=\int_{\mathbb{H}^n} I_{0|t}^\alpha(u^p)(t,\eta)\Theta(\eta)\,\mathrm{d\eta}+ \int_0^t \int_{\mathbb{H}^n}\frac{d}{\mathrm{dt}}\left(S_{\mathbb{H}}(t-s)I_{0|s}^\alpha(u^p)(s,\eta)\right)\Theta(\eta)\,\mathrm{d\eta}\,ds\nonumber\\
		&&=\int_{\mathbb{H}^n} I_{0|t}^\alpha(u^p)(t,\eta)\Theta(\eta)\,\mathrm{d\eta}+ \int_0^t \int_{\mathbb{H}^n}\Delta_{\mathbb{H}}\left(S_{\mathbb{H}}(t-s)I_{0|s}^\alpha(u^p)(s,\eta)\right)\Theta(\eta)\,\mathrm{d\eta}\,ds\nonumber\\
		&&=\int_{\mathbb{H}^n} I_{0|t}^\alpha(u^p)(t,\eta)\Theta(\eta)\,\mathrm{d\eta}+ \int_0^t \int_{\mathbb{H}^n}S_{\mathbb{H}}(t-s)I_{0|s}^\alpha(u^p)(s,\eta)\Delta_{\mathbb{H}}\Theta(\eta)\,\mathrm{d\eta}\,ds,
	\end{eqnarray}
	where we have used $I_{0|t}^\alpha (u^p)\in C([0,T];C_0(\mathbb{H}^n)).$\\
	Thus, using \eqref{IEE}, \eqref{int1A} and \eqref{int2A}, we conclude that \eqref{differentiation} implies
	\[
	\frac{d}{\mathrm{dt}} \ \int_{\mathbb{H}^n}u(t,\eta)\Theta(\eta)\,\mathrm{d\eta}=  \int_{\mathbb{H}^n}u(t,\eta)\Delta_{\mathbb{H}}\Theta(\eta)\,\mathrm{d\eta}+\int_0^t(t-s)^{-\gamma}\int_{\mathbb{H}^n} u^p(s,\eta)\Theta(\eta)\,\mathrm{d\eta}\,ds.
	\]
	Applying Lemma \ref{lemma1}, we choose $c>0$ so that the test function $\Theta$ is normalized, i.e.,
	$$ \int_{\mathbb{H}^n}\Theta(\eta)\,\mathrm{d\eta}=1,$$
	with $\varepsilon=1/(2Q+4)$. Under this choice, we have
	$$\Delta_{\mathbb{H}}\Theta(\eta)\geq-\Theta(\eta).$$ 
Therefore, by letting 
	$$f(t):= \displaystyle\int_{\mathbb{H}^n}u(t,\eta)\Theta(\eta)\,\mathrm{d\eta},$$ we arrive that
	$$f'(t)\geq -f(t)+\int_0^t(t-s)^{-\gamma}\int_{\mathbb{H}^n} u^p(s,\eta)\Theta(\eta)\,\mathrm{d\eta}\,ds.$$
	As $\displaystyle \int_{\mathbb{H}^n}\Theta (\eta)\,\mathrm{d\eta}=1$, by Jensen's inequality, we conclude that
	\begin{equation}\label{int3A}
		f'(t)+f(t)\geq \int_0^t(t-s)^{-\gamma}f^p(s)\,ds,\qquad\hbox{for all}\,\,t\in[0,T].
	\end{equation}
	Since $u_0\geq0$, $u_0\not\equiv0$ implies $f(0)>0$, \eqref{int3A}, $p\gamma\leq 1$, and Lemma \ref{lemma2} $\mathrm{(ii)}$ yield a contradiction. 
\end{proof}

\begin{proof}[The case $p\leq p_\gamma$] The proof is also by contradiction. Suppose that $u$ is a global mild solution to \eqref{1}, then $u$ is a global weak solution and therefore
 \begin{equation}\label{newweaksolution}
\int_{\mathbb{H}^n}u_0(\eta)\varphi(0,\eta)\,d\eta+\Gamma(\alpha)\int_0^T\int_{\mathbb{H}^n}I_{0|t}^\alpha(u^p)\varphi(t,\eta)\,d\eta\,dt=-\int_0^T\int_{\mathbb{H}^n}u\left[\partial_t\varphi+\Delta_{\mathbb{H}}\varphi \right]\,d\eta\,dt,
\end{equation}
 for all $T>0$ and all compactly supported $\varphi\in C^{2,1}([0,T]\times\mathbb{H}^n)$ such that $\varphi(T,\cdotp)=0$, where $\alpha:=1-\gamma\in(0,1)$.
 Let $R$ and $T$ be large parameters in $(0,\infty)$. Let us choose 
	$$\varphi(t,\eta)=D^\alpha_{t|T}\left(\widetilde{\varphi}(t,\eta)\right):=D^\alpha_{t|T}\left(w_1(t)\varphi_R(\eta)\right),\qquad \eta\in\mathbb{H}^{n},\,\,t\geq 0,$$
	where $w_1$ is given by \eqref{w1} and $\varphi_R$ be a test function as defined in \eqref{varphi_R}.  Consequently, for every $T>R$, 
  \begin{eqnarray}\label{weak1A}
&{}&D^\alpha_{t|T}w_1(0)\int_{\mathbb{H}^n}u_0(\eta)\varphi_R(\eta)\,d\eta
+\Gamma(\alpha)\int_0^T\int_{\mathbb{H}^n}I_{0|t}^\alpha(u^p)D^\alpha_{t|T}\left(\widetilde{\varphi}(t,\eta)\right)\,d\eta\,dt\nonumber
\\
&{}&=-\int_0^T\int_{\mathbb{H}^n}u(t,\eta)\partial_tD^\alpha_{t|T}\left(\widetilde{\varphi}(t,\eta)\right)\,d\eta\,dt-\int_0^T\int_{\mathbb{H}^n}u(t,\eta)\Delta_{\mathbb{H}}D^\alpha_{t|T}\left(\widetilde{\varphi}(t,\eta)\right)\,d\eta\,dt.
\end{eqnarray}
Therefore, using \eqref{IP}, \eqref{I4}, and \eqref{P1} in  \eqref{weak1A}, we obtain
  \begin{eqnarray*}
&{}&C\,T^{-\alpha}\int_{\mathbb{H}^n}u_0(\eta)\varphi_R(\eta)\,d\eta
+\Gamma(\alpha)\int_0^T\int_{\mathbb{H}^n}D^{\alpha}_{0|T}I_{0|t}^\alpha(u^p)\widetilde{\varphi}(t,\eta)\,d\eta\,dt
\\
&{}&=\int_0^T\int_{\mathbb{H}^n}u(t,\eta)D^{1+\alpha}_{t|T}\left(\widetilde{\varphi}(t,\eta)\right)\,d\eta\,dt-\int_0^T\int_{\mathbb{H}^n}u(t,\eta)\Delta_{\mathbb{H}}D^\alpha_{t|T}\left(\widetilde{\varphi}(t,\eta)\right)\,d\eta\,dt.
\end{eqnarray*}
Moreover, using \eqref{I3} and Lemma \eqref{lem:testfunction}, we may write
 \begin{eqnarray}\label{weak2A}
&{}&C\,T^{-\alpha}\int_{\mathbb{H}^n}u_0(\eta)\varphi_R(\eta)\,d\eta
+\Gamma(\alpha)\int_0^T\int_{\mathbb{H}^n}u^p(t,\eta)\widetilde{\varphi}(t,\eta)\,d\eta\,dt\nonumber
\\
&{}&=\int_0^T\int_{\mathbb{H}^n}u(t,\eta)D^{1+\alpha}_{t|T}\left(\widetilde{\varphi}(t,\eta)\right)\,d\eta\,dt-\int_0^T\int_{\mathbb{H}^n}u(t,\eta)\Delta_{\mathbb{H}}\varphi_R(\eta)D^\alpha_{t|T}w_1(t)\,d\eta\,dt\nonumber\\
&{}&\leq \int_0^T\int_{\mathbb{H}^n}u(t,\eta)|D^{1+\alpha}_{t|T}\left(\widetilde{\varphi}(t,\eta)\right)|\,d\eta\,dt+CR^{-1}\int_0^T\int_{\mathbb{H}^n}u(t,\eta)(\varphi_R^*)^{1/p}D^\alpha_{t|T}w_1(t)\,d\eta\,dt\nonumber\\
&{}&=:I_1+I_2.
\end{eqnarray}
 To estimate $I_1$, we rewrite the integrand by inserting the factor $\widetilde{\varphi}^{1/p}\widetilde{\varphi}^{-1/p}$ and then apply Young's inequality, which leads to
	\begin{align}\label{I1A}
		I_1\leqslant\frac{\Gamma(\alpha)}{4} \int_0^T\int_{\mathbb{H}^n}(u(t,\eta))^{p}\widetilde{\varphi}(t,\eta)\,\mathrm{d\eta}\,\mathrm{dt}+C\int_0^T\int_{\mathbb{H}^n}\varphi_R(\eta)(w_1(t))^{-p'/p}\left|D^{1+\alpha}_{t|T} w_1(t)\right|^{p'}\mathrm{d}\eta\,\mathrm{d}t.
	\end{align}
To proceed with the estimation of $I_2$,  we rewrite the integrand by introducing the factor $(w_1\varphi_R^*)^{1/p}(w_1\varphi_R^*)^{-1/p}$. This allows the application of Young's inequality, along with the bound $\varphi_R^*\leq \varphi_R$, leading to the estimate 
	\begin{equation}\label{I2A}
I_2\leqslant\frac{\Gamma(\alpha)}{4} \int_0^T\int_{\mathbb{H}^n}|u(t,\eta)|^{p}\widetilde{\varphi}(t,\eta)\,\mathrm{d\eta}\,\mathrm{dt}+C\,R^{-p'}\int_0^T\int_{\mathbb{H}^n} \mathbbm{1}_{[\frac{1}{2},1]}\left(\xi_R(\eta)\right)(w_1(t))^{-p'/p}\left|D^{\alpha}_{t|T} w_1(t)\right|^{p'}\mathrm{d}\eta\,\mathrm{d}t.
\end{equation}
	Combining \eqref{weak2A}, \eqref{I1A}, and \eqref{I2A}, we arrive at
	\begin{eqnarray*}\label{weak2A*}
		&{}&C\,T^{-\alpha}\int_{\mathbb{H}^n}u_0(\eta)\varphi_R(\eta)\,d\eta
+\frac{\Gamma(\alpha)}{2}\int_0^T\int_{\mathbb{H}^n}u^p(t,\eta)\widetilde{\varphi}(t,\eta)\,d\eta\,dt\nonumber\\
		&{}&\leq C\int_0^T\int_{\mathbb{H}^n}\varphi_R(\eta)(w_1(t))^{-p'/p}\left|D^{1+\alpha}_{t|T} w_1(t)\right|^{p'}\mathrm{d}\eta\,\mathrm{d}t\nonumber\\
		&{}&\quad+\,CR^{-p'}\int_{\mathbb{H}^n} \mathbbm{1}_{[\frac{1}{2},1]}\left(\xi_R(\eta)\right)\,\mathrm{d}\eta\int_0^T(w_1(t))^{-p'/p}\left|D^{\alpha}_{t|T} w_1(t)\right|^{p'}\,\mathrm{d}t.
			\end{eqnarray*}
	At this stage, by introducing the scaled variables 
	$$\widetilde{x}=\frac{x}{R^{\frac{1}{2}}},\qquad\widetilde{y}=\frac{y}{R^{\frac{1}{2}}},\qquad\widetilde{\tau}=\frac{\tau}{R},\qquad \widetilde{t}=\frac{t}{T},\qquad\hbox{for}\,\,\eta=(x,y,\tau)\in\mathbb{H}^n,\,\,t>0,$$
	and using Lemma \ref{L1} and $u_0\geq0$, we get
	\begin{equation}\label{weak3A}
		\int_0^T\int_{\mathbb{H}^n}u^p(t,\eta)\widetilde{\varphi}(t,\eta)\,d\eta\,dt\lesssim T^{1-(1+\alpha)p'}R^{\frac{Q}{2}}+\,T^{1-\alpha p'}R^{-p'+\frac{Q}{2}}.
	\end{equation}
	We consider two separate cases. For $p<p_\gamma$, choosing $R=T$, it follows from \eqref{weak3A} that
	\begin{equation}\label{weak4A}
			\int_0^T\int_{\mathbb{H}^n}u^p(t,\eta)\widetilde{\varphi}(t,\eta)\,d\eta\,dt\lesssim T^{-\delta}
	\end{equation}
	where $$\delta:=-1+(1+\alpha)p'-\frac{Q}{2}>0.$$ Passing to the limit in \eqref{weak4A} as $T\rightarrow\infty$, and using the monotone convergence theorem, we get
	$$0\leq \int_0^\infty\int_{\mathbb{H}^n}(u(t,\eta))^{p}\,\mathrm{d\eta}\,\mathrm{dt}\leq 0.$$
	Therefore $u\equiv 0$; contradiction. In order to get a contradiction in the case $p=p_\gamma$ too, we need to improve the estimate of $I_2$ by using H\"older's inequality instead of Young's inequality. Indeed,
		\begin{equation}\label{I2B}
I_2\lesssim R^{-1}\left(\int_0^T\int_{\mathbb{H}^n}(u(t,\eta))^{p}w_1(t)\varphi_R^*(\eta)\,\mathrm{d\eta}\,\mathrm{dt}\right)^{\frac{1}{p}}\left(\int_0^T\int_{\mathbb{H}^n} \mathbbm{1}_{[\frac{1}{2},1]}\left(\xi_R(\eta)\right)(w_1(t))^{-p'/p}\left|D^{\alpha}_{t|T} w_1(t)\right|^{p'}\mathrm{d}\eta\,\mathrm{d}t\right)^{\frac{1}{p'}}.
\end{equation}
Substituting \eqref{I1A} and \eqref{I2B} into \eqref{weak2A}, and noting that $u_0\geq0$, we arrive at
\begin{eqnarray}\label{weak2B}
		&{}&\int_0^T\int_{\mathbb{H}^n}u^p(t,\eta)\widetilde{\varphi}(t,\eta)\,d\eta\,dt\\
		&{}&\lesssim \int_0^T\int_{\mathbb{H}^n}\varphi_R(\eta)(w_1(t))^{-p'/p}\left|D^{1+\alpha}_{t|T} w_1(t)\right|^{p'}\mathrm{d}\eta\,\mathrm{d}t\nonumber\\
		&{}&+ R^{-1}\left(\int_0^T\int_{\mathbb{H}^n}(u(t,\eta))^{p}w_1(t)\varphi_R^*(\eta)\,\mathrm{d\eta}\,\mathrm{dt}\right)^{\frac{1}{p}}\left(\int_0^T\int_{\mathbb{H}^n} \mathbbm{1}_{[\frac{1}{2},1]}\left(\xi_R(\eta)\right)(w_1(t))^{-p'/p}\left|D^{\alpha}_{t|T} w_1(t)\right|^{p'}\mathrm{d}\eta\,\mathrm{d}t\right)^{\frac{1}{p'}}.\nonumber
					\end{eqnarray}
			By applying the same change of variables as before and by setting $R=TK^{-1}$, where $K\ge 1$ and $K<T$, so that $T$ and $K$ cannot simultaneously tend to infinity, and considering the fact that $p=p_\gamma$, we obtain
			\begin{equation}\label{weak2C}
	\int_0^T\int_{\mathbb{H}^n}u^p(t,\eta)\widetilde{\varphi}(t,\eta)\,d\eta\,dt\lesssim K^{-\frac{Q}{2}} +\,K^{1-\frac{Q}{2p'}}\left(\int_0^T\int_{\mathbb{H}^n}(u(t,\eta))^{p}w_1(t)\varphi_R^*(\eta)\,\mathrm{d\eta}\,\mathrm{dt}\right)^{\frac{1}{p}}.
	\end{equation}
	On the other hand, from \eqref{weak4A} as $T\rightarrow\infty$, and taking into account that $p=p_\gamma$, it follows that
	\begin{equation}\label{regularityA}
		u\in L^p((0,\infty),L^p(\mathbb{H}^{n})).
	\end{equation}
	Taking the limit as $T\rightarrow\infty$ in \eqref{weak2C}, and applying \eqref{regularityA} along with the Lebesgue dominated convergence theorem, we conclude that
	$$\int_0^\infty\int_{\mathbb{H}^n}u^p(t,\eta)\,d\eta\,dt\lesssim K^{-\frac{Q}{2}}.$$
	Therefore, taking a sufficiently large $K$ we obtain the desired contradiction. The proof is complete.
\end{proof}
\section{Global existence}\label{sec5}
We now proceed to prove the global existence result stated in Theorem \ref{Blowup-global}-$(\mathrm{ii})$.
\begin{proof}[Proof of Theorem \ref{Blowup-global}-$(\mathrm{ii})$] Since $p>p_c$, it is possible to choose a positive constant $q>0$ such that
\begin{equation}\label{estiA}
    \frac{2-\gamma}{p-1}-\frac{1}{p}<\frac{Q}{2
    q}<\frac{1}{p-1}\quad \text{with} \quad q\geq p.
\end{equation}
From this, it follows that 
\begin{equation}\label{estiB}
    q>\frac{Q(p-1)}{2}>q_{\mathrm{sc}}>1.
\end{equation}
We then define
\begin{equation}\label{estiC}
    \beta:=\frac{Q}{2 q_{\mathrm{sc}}}-\frac{Q}{2
    q}=\frac{2-\gamma}{p-1}-\frac{Q}{2
    q}.
\end{equation}
Therefore, based on relations (\ref{estiA})-(\ref{estiC}), the following relations hold
\begin{equation}\label{estiD}
    \beta>\frac{1-\gamma}{p-1}>0,\quad \frac{Q(p-1)}{2
    q}+(p-1)\beta+\gamma=2,\quad\hbox{and}\quad p\beta<1.
\end{equation}
 Since $u_0\in L^{q_{\mathrm{sc}}}(\mathbb{H}^{n})$, applying Lemma \ref{Lp-Lqestimate} with $q>q_{\mathrm{sc}}$, and using \eqref{estiC}, we
get
\begin{equation}\label{estiE}
    \sup_{t>0}t^\beta\|S_{\mathbb{H}}(t)u_0\|_{L^q}\leq
   C \|u_0\|_{L^{q_{\mathrm{sc}}}}=:\rho<\infty.
\end{equation}
Set
\begin{equation}\label{estiF}
    \mathbb{X}:=\left\{u\in
    L^\infty((0,\infty),L^q(\mathbb{H}^{n}));\;\sup_{t>0}t^\beta\|u(t)\|_{L^q}\leq\delta\right\},
\end{equation}
where $\delta>0$ is chosen to be sufficiently small. For $u,v\in \mathbb{X}$, we define the metric
\begin{equation}\label{estiG}
    d_{\mathbb{X}}(u,v):=\sup_{t>0}t^\beta\|u(t)-v(t)\|_{L^q}.
\end{equation}
It is straightforward to verify that $(\mathbb{X},d)$ is a nonempty complete metric space. For $u\in \mathbb{X}$, we define the mapping $\Phi(u)$ by
\begin{equation}\label{estiH}
    \Phi(u)(t):=S_{\mathbb{H}}(t)u_0 + \int_{0}^{t}S_{\mathbb{H}}(t-s) \int_0^s(s-\tau)^{-\gamma}(u(\tau))^{p} \, \mathrm{d}\tau \,\mathrm{d}s,\quad \text{for all}\,\, t\geq0.
\end{equation}
Let us now verify that the operator $\Phi:  \mathbb{X} \rightarrow \mathbb{X}$. By employing the assumption $q\geq p$, inequalities \eqref{estiE} and \eqref{estiF}, along with Lemma \ref{Lp-Lqestimate}, we derive the following estimate for any $u \in  \mathbb{X}$, 
\begin{eqnarray}\label{estiI}
    t^\beta\|\Phi(u)(t)\|_{L^q}&\leq& \rho+\,Ct^\beta\int_0^t(t-s)^{-\frac{Q(p-1)}{2q}}
    \int_0^s(s-\sigma)^{-\gamma}\|u(\sigma)\|^{p}_{L^{q}}\,d\sigma\,ds\nonumber\\
    &\leq&\rho+C\delta^{p} t^\beta\int_0^t\int_0^s(t-s)^{-\frac{Q(p-1)}{2
    q}}(s-\sigma)^{-\gamma}\sigma^{-\beta p}\,d\sigma\,ds.
\end{eqnarray}
Now, using the parameter constraints in \eqref{estiA} and \eqref{estiD}, together with the condition $p\beta<1$, the double integral becomes
\begin{equation}\label{estiJ}\begin{split}
 \int_0^t\int_0^s(t-s)^{-\frac{Q(p-1)}{2
    q}}(s-\sigma)^{-\gamma}\sigma^{-\beta p}\,d\sigma\,ds&= C\int_0^t(t-s)^{-\frac{Q(p-1)}{\beta
    q}}s^{1-\gamma-\beta p}\,ds\\&=C t^{-\beta},\end{split}
\end{equation}
valid for all $t\geq0$. It then follows from estimates \eqref{estiI} and \eqref{estiJ} that
\begin{equation}\label{estiK}
    t^\beta\|\Phi(u)(t)\|_{L^q}\leq \rho+C\delta^p.
\end{equation}
Hence, if $\rho$ and $\delta$ are chosen sufficiently small such that $\rho+C\delta^p\leq\delta$, it follows that $\Phi(u)\in\mathbb{X}$, i.e., $\Phi: \mathbb{X}\rightarrow \mathbb{X}$. A similar argument shows that, under the same smallness assumptions on $\rho$ and $\delta$, the operator $\Phi$ is a strict contraction. Therefore, it has a unique fixed point $u\in \mathbb{X}$ which corresponds to a mild solution of equation \eqref{1}.

We now aim to prove that $u\in C([0,\infty),C_0(\mathbb{H}^{n}))$. We begin by showing that $u\in C([0,T],C_0(\mathbb{H}^{n}))$ for some sufficiently small $T>0$. Indeed, the previous argument ensures uniqueness in the space $ \mathbb{X}_T,$ where for any $T>0,$
$$
 \mathbb{X}_T:=\left\{u\in
    L^\infty((0,T),L^q(\mathbb{H}^{n}));\;\sup_{0<t<T}t^\beta\|u(t)\|_{L^q}\leq\delta\right\}.
$$
Let $\tilde{u}$ denote the local solution of \eqref{1} established in Theorem \ref{Local}. From inequality \eqref{estiB}, we know that $p_{\mathrm{sc}}<q<\infty$, which implies 
$$u_0\in C_0(\mathbb{H}^{n})\cap L^{q_{\mathrm{sc}}}(\mathbb{H}^{n})\subset C_0(\mathbb{H}^{n})\cap L^{q}(\mathbb{H}^{n}).$$ 
Thus,  Theorem \ref{Local} guarantees that $\tilde{u}\in C([0,T_{\max}),C_0(\mathbb{H}^{n})\cap L^q(\mathbb{H}^{n}))$. Consequently, for sufficiently small $T>0$, we have
$$\sup\limits_{t\in(0,T)}t^\beta\|\tilde{u}(t)\|_{L^q}\leq\delta.$$
By the uniqueness of solutions, it follows that $u=\tilde{u}$ on
$[0,T]$, leading to the conclusion that $u\in C([0,T],C_0(\mathbb{H}^{n}))$.

To extend the regularity to $[T,\infty)$, we employ a bootstrap argument. For $t>T,$ we express $u(t)$ as
\begin{eqnarray*}
  u(t)-S_{\mathbb{H}}(t)u_0 &=&
  \int_0^tS_{\mathbb{H}}(t-s)\int_0^T
    (s-\sigma)^{-\gamma}(u(\sigma))^p\,d\sigma\,ds+\int_0^tS_{\mathbb{H}}(t-s)\int_T^s(s-\sigma)^{-\gamma}(u(\sigma))^p\,d\sigma\,ds\\
   &\equiv& I_1(t)+I_2(t).
\end{eqnarray*}
Since $u\in C([0,T],C_0(\mathbb{H}^{n}))$, we immediately have $I_1\in
C([T,\infty),C_0(\mathbb{H}^{n})).$ Furthermore, by similar estimates used in the fixed-point argument and noting that $$t^{-\beta}\leq
T^{-\beta}<\infty,$$ we also get $I_1\in C([T,\infty),L^q(\mathbb{H}^{n}))$.
Next, from \eqref{estiB}, we observe that $q>Q(p-1)/2$, which guarantees the existence of some $r\in(q,\infty]$ such that
\begin{equation}\label{estiL}
\frac{Q}{2}\left(\frac{p}{q}-\frac{1}{r}\right)<1.
\end{equation}
For $T<s<t$, since $u\in
L^\infty((0,\infty),L^q(\mathbb{H}^{n})),$ we obtain
$$u^p\in L^\infty((T,s),L^{\frac{q}{p}}(\mathbb{H}^{n})).$$
 Thus, applying Lemma \ref{Lp-Lqestimate} and condition \eqref{estiL}, we deduce that $I_2\in
C([T,\infty),L^r(\mathbb{H}^{n}))$.
Since the terms $S_{\mathbb{H}}(t)u_0$ and $I_1$ belong to
$$C([T,\infty),C_0(\mathbb{H}^{n}))\cap
C([T,\infty),L^q(\mathbb{H}^{n})),$$ we conclude that $u\in
C([T,\infty),L^r(\mathbb{H}^{n}))$. Iterating this procedure a finite
number of times, eventually leads to $u\in
C([T,\infty),C_0(\mathbb{H}^{n})).$ 
\end{proof}

\section{Lifespan estimates}\label{sec6}
\begin{proof}[Proof of Theorem \ref{Lifespan}]
By \eqref{weak2A}-\eqref{weak3A}, and setting $R=T$, it follows that
$$\varepsilon\,T^{-\alpha}\int_{\mathbb{H}^n}u_0(\eta)\varphi_R(\eta)\,d\eta\lesssim T^{1-(2-\gamma)p'+\frac{Q}{2}}.$$
that is,
\begin{equation}\label{corC}
\varepsilon\int_{\mathbb{H}^n}u_0(\eta)\varphi_R(\eta)\,d\eta\lesssim T^{(2-\gamma)(1-p')+\frac{Q}{2}}.
\end{equation}
Furthermore, by applying the dominated convergence theorem, we have
$$\lim_{R \to \infty} \int_{\mathbb{H}^n}u_0(\eta)\varphi_R(\eta)\,d\eta=\int_{\mathbb{H}^n} u_0(\eta)\,d\eta>0.$$
This implies that
$$  1\lesssim \int_{\mathbb{H}^n}u_0(\eta)\varphi_R(\eta)\,d\eta,\qquad\text{for}\,\, T \geq T_1,$$
 where \( T_1 \) is a suitably large constant. Hence, for all \( T \geq T_1 \), we conclude that
$$\varepsilon\lesssim T^{(2-\gamma)(1-p')+\frac{Q}{2}}.$$
As $p < p_{\mathrm{sc}}$, the power for $T$ is negative in the last estimate. Thus,
$$T\lesssim \varepsilon^{-\left(\frac{2-\gamma}{p-1}-\frac{Q}{2}\right)^{-1}}.$$
Note that we have assumed without loss of generality that \( T\geq T_1 \). Indeed, if \( T \leq T_1 \), then for sufficiently small \( \varepsilon > 0 \), the inequality $$T \lesssim \varepsilon^{-\left(\frac{2-\gamma}{p-1} - \frac{Q}{2}\right)^{-1}}$$ is trivially satisfied. Since the statement holds for any $T\in(0,T_\varepsilon)$, the same conclusion is true for $T_\varepsilon$.
The proof is complete.
\end{proof}
\begin{proof}[Proof of Theorem \ref{Lifespan-2}]
By exploiting  \eqref{corA} and \eqref{varphi_R}, we have
\begin{equation}\label{corB}\begin{split}
 \varepsilon \int_{\mathbb{H}^n}u_0(\eta)\varphi_T(\eta)\,d\eta&\ge  \varepsilon \int_{\frac{T}{4}\leq|\eta|^2\leq \frac{T}{2}}u_0(\eta)\,d\eta\\&\ge  \varepsilon \int_{\frac{T}{4}\leq|\eta|^2\leq \frac{T}{2}} \frac{1}{(1+|\eta|)^{\kappa}}\,d\eta
 \\&\geq C\varepsilon\,T^{\frac{Q-\kappa}{2}},\end{split}\end{equation}
for any $T>1$.  By combining \eqref{corC} and \eqref{corB}, we conclude that
 $$
\varepsilon\,T^{\frac{Q-\kappa}{2}}\lesssim T^{(2-\gamma)(1-p')+\frac{Q}{2}}, \qquad  \hbox{for all}\,\,T>1,
$$
that is,
$$
\varepsilon\lesssim T^{\frac{\kappa}{2}-(2-\gamma)(p'-1)}, \qquad  \hbox{for all}\,\,T>1,
$$
  which leads, using $\kappa< 2(2-\gamma)/(p-1)$, to derive that there exists a constant $\varepsilon_0> 0$ such
that $T_{\varepsilon}$ satisfies
$$
 T_{\varepsilon}\le
 C\varepsilon^{-\left(\frac{2-\gamma}{p-1}-\frac{\kappa}{2}\right)^{-1}},\quad\hbox{ for all}\,\,\varepsilon\leq\varepsilon_0.
 $$ The proof is complete.
\end{proof}


\section*{Acknowledgment}
Ahmad Z. Fino is supported by the Research Group Unit, College of Engineering and Technology, American University of the Middle East. Z. Sabbagh would like to express her gratitude to A. Z. Fino for hosting her in Kuwait where she initiated this work. Berikbol T. Torebek is supported by the Science Committee of the Ministry of Education and Science of the Republic of Kazakhstan (Grant No. AP23483960) and by the Methusalem programme of the Ghent University Special Research Fund (BOF) (Grant number 01M01021).

\section*{Declaration of competing interest}
	The Authors declare that there is no conflict of interest

\section*{Data Availability Statements} The manuscript has no associated data.

\end{document}